\documentclass[a4paper,11pt]{amsart}
\usepackage{latexsym, delarray}
\usepackage{amssymb}

\usepackage{enumerate}
\usepackage[colorlinks=true]{hyperref}

\newcommand{\A}{\mathbf{A}}
\newcommand{\C}{\mathbf{C}}
\newcommand{\cond}[1]{\mathfrak{c}_{#1}}
\newcommand{\cyclomod}[1][l]{\chi_{#1}}
\newcommand{\eps}{\varepsilon}
\newcommand{\f}{\mathrm{f}}
\newcommand{\Ff}[1][l]{\mathbf{F}_{#1}}
\newcommand{\Ffbar}[1][l]{\overline{\mathbf{F}}_{#1}}
\newcommand{\Frob}{\mathrm{Frob}}
\newcommand{\G}{\mathrm{G}}
\newcommand{\Gal}{\mathrm{Gal}}
\newcommand{\GQ}{\mathrm{Gal}(\Qbar/\Q)}
\newcommand{\GL}{\mathrm{GL}}
\newcommand{\h}{\mathfrak{H}}
\newcommand{\im}{\mathrm{Im}}
\newcommand{\ord}[1]{\mathrm{ord}_{#1}}
\newcommand{\place}{w}
\newcommand{\Q}{\mathbf{Q}}
\newcommand{\Qbar}{\overline{\mathbf{Q}}}
\newcommand{\R}{\mathbf{R}}
\newcommand{\SL}{\mathrm{SL}}
\newcommand{\trivial}{\mathbf{1}}
\newcommand{\Z}{\mathbf{Z}}
\newcommand{\Zbar}{\overline{\mathbf{Z}}}

\theoremstyle{plain}
\newtheorem{theorem}{Theorem}

\newtheorem{prop}[theorem]{Proposition}
\newtheorem{lemma}[theorem]{Lemma}
\newtheorem{cor}[theorem]{Corollary}

\newtheorem*{claim*}{Claim}

\theoremstyle{remark}
\newtheorem{question}{Question}
\newtheorem{rk}[question]{Remark}

\theoremstyle{definition}

\begin{document}

\title{Strong modularity of reducible Galois representations}
\author[Nicolas Billerey]{Nicolas Billerey $^\P$}
\thanks{\P\ Partially supported by CNRS and ANR-14-CE-25-0015 Gardio.}
\address{(1) Universit\'e Clermont Auvergne, Universit\'e Blaise Pascal,
	Laboratoire de Math\'ematiques,
	BP 10448,
	F-63000 Clermont-Ferrand, France.
	(2) CNRS, UMR 6620, LM, F-63171 Aubi\`ere, France}
\email{Nicolas.Billerey@math.univ-bpclermont.fr}

\author[Ricardo Menares]{Ricardo Menares $^\dag$ }
\address{Instituto de Matem\'aticas, Pontificia Universidad Cat\'olica de Valpara\'iso, Blanco Viel 596, Cerro Bar\'on, Valpara\'iso, Chile}
\thanks{\dag\ Pontificia Universidad Cat\'olica de Valpara\'iso. Partially supported by PUCV grant 037.469/2015}
\email{ricardo.menares@pucv.cl}
\date{\today}


\begin{abstract}
Let $\rho\colon\GQ\rightarrow\GL_2(\Ffbar)$ be an odd, semi-simple Galois representation. Here, $l\geq5$ is prime and $\Ffbar$ is an algebraic closure of the finite field $\Z/l\Z$. When the representation is irreducible, the strongest form of Serre's  original modularity conjecture  (which is now proved) asserts that \(\rho\) arises from a cuspidal eigenform of type~$(N,k,\eps)$ over~$\Ffbar$, where $N$, $k$ and $\eps$ are, respectively, the level, weight and character attached to~$\rho$ by Serre.

In this paper we characterize, under the assumption $l>k+1$, reducible semi-simple representations, that we call strongly modular, such that the same result holds. This characterization generalizes a classical theorem of Ribet pertaining to the case $N=1$. When the representation is not strongly modular, we give a necessary and sufficient condition on primes $p$ not dividing~$Nl$ for which $\rho$ arises in level $Np$, hence genera\-lizing a classical theorem of Mazur concerning the case~$(N,k)=(1,2)$. 

The  proofs rely on the classical analytic  theory of Eisenstein series and on local properties of automorphic representations attached to newforms.

\end{abstract}

\maketitle


\section*{Introduction}

 \bigskip

Let~\(l\) be a prime number. We denote by $\Ffbar$ and $\Qbar$ algebraic closures of $\Ff=\Z/l\Z$ and the rational field $\Q$ respectively. In this article we are interested in Galois representations of the form

\begin{equation}\label{rep}
\rho\colon\Gal(\Qbar/\Q)\longrightarrow\GL_2(\Ffbar),
\end{equation}

\noindent where $\rho$ is a continuous homomorphism. Let $N\geq1$ and $k\geq2$ be two integers with $N$ coprime to~$l$ and let $\eps:(\Z/N\Z)^\times\rightarrow\Ffbar^\times$ be a character. Let~$f$ be a cusp form of type~$(N,k,\eps)$ over~$\Ffbar$ (in the sense of~\cite[D\'ef. p.~193]{Ser87}) which is an eigenfunction for the $p$-th Hecke operator with eigenvalue $a_p$ in~$\Ffbar$ for each prime number~$p$. By work of Deligne, to such a form~$f$, one can attach a (unique up to isomorphism) semi-simple odd Galois representation~$\rho_f$ which is unramified outside $Nl$ and satisfies the following property~:  If $\Frob_p$ denotes a Frobenius element at a prime~ $p\nmid Nl$, then the characteristic polynomial of $\rho_f(\Frob_p)$ is given by
$$
X^2-a_pX+\varepsilon(p)p^{k-1}.
$$
According to a standard terminology,  a Galois representation $\rho$ is called \emph{modular} if it is isomorphic to~$\rho_f$ for some~$f$ as above. In that case, we also say that $\rho$ arises from~$f$.

Moreover, to any given Galois representation~$\rho$, Serre attaches in~\cite[\S\S1-2]{Ser87}  a triple $(N,k,\varepsilon)$, which we refer to as the Serre type of~$\rho$, consisting of an integer $N\geq 1$ coprime to~$l$, an integer~$k\geq2$ and a group homomorphism $\varepsilon\colon(\Z/N\Z)^{\times}\rightarrow\Ffbar^{\times}$ which are called the conductor, weight and character of~$\rho$ respectively.

In this paper, we shall say that a Galois representation $\rho$ is \emph{strongly \mbox{modular}} if it arises from a cuspidal eigenform $f$ over $\Ffbar$ of type~$(N,k,\eps)$ where $(N,k,\varepsilon)$ is the Serre type of~$\rho$.

With this terminology, the strong form (\cite[(3.2.4$_?$)]{Ser87}) of Serre's modu\-larity conjecture, asserts that any odd, irreducible Galois representation~$\rho$ as in \eqref{rep},  with $l\geq5$, is strongly modular.  This conjecture has now been proved through the combined work of many mathematicians (see \cite{KhWi09a,KhWi09b} and the references therein). 

We remark that results of Carayol (\emph{cf.} \cite[Thm.~(A)]{Car86} and the considerations in  \cite[{\bf 1.-2.}]{Car89}),  ensure  that whenever $\rho$ is strongly modular, the eigenform $f$ can be taken to be the reduction of a  newform $F$ (in characteristic zero) of level $N$.

In this article, we address the case where $\rho$ is reducible. Let
\[
\nu_1,\nu_2\colon\GQ\longrightarrow\Ffbar^\times
\]
be continuous characters and assume that $\rho=\nu_1\oplus\nu_2$ defines an odd (semi-simple) Galois representation of Serre type~\((N,k,\eps)\). Then,  $\rho$ is modular (\emph{e.g.} see~\cite[Thm.~2.1]{BiMe15}) but need not be strongly modular. Our task is to provide a necessary and sufficient condition for such a reducible Galois representation to be strongly modular. Thanks to Ribet, such a characterization is known in the case $N=1$ under the assumption $l>k+1$ (see \cite[Lem.~5.2]{Rib75} or~\cite[Cor.~3.7]{BiMe15} for a reformulation  in this context). Under the same assumption, we prove in this paper a generalization of this result to arbitrary conductors. 

Let~$\eta \colon\GQ\rightarrow\Ffbar^\times$ be a character unramified at~$l$. For any integer~$k\geq2$ satisfying $l>k+1$, we define in paragraph~\ref{ss:Bernoulli_mod_l} a mod $l$  Bernoulli number~$B_{k,\eta }\in\Ffbar$ associated with~$\eta $ (our $B_{k,\eta}$ is essentially the reduction of a classical $k$-th Bernoulli number attached to a lift of $\eta$, but some care has to be taken due to denominators and the choice of place). For every prime number~$p$, set
\[
\eta(p)=\left\{
\begin{array}{ll}
\eta(\Frob_p) & \text{if $\eta$ is unramified at~$p$} \\
0 & \text{if $\eta$ is ramified at~$p$.}
\end{array}
\right.
\]
In this notation, the following is the main result of the paper.
\begin{theorem}\label{main}
Let $\nu_1,\nu_2\colon\GQ\rightarrow\Ffbar^\times$ be characters defining an odd (semi-simple) Galois representation $\rho=\nu_1\oplus\nu_2$ of Serre type~$(N,k,\eps)$ with $l>k+1$. Then, there exist characters \(\eps_1,\eps_2\colon\GQ\rightarrow\Ffbar^{\times}\) unramified at~\(l\) such that $\rho=\eps_1\oplus\eps_2\cyclomod[l]^{k-1}$, where $\cyclomod[l]$ is the mod~$l$ cyclotomic character. 
Set $\eta =\eps_1^{-1}\eps_2$. The representation~$\rho$ is strongly modular if and only if
\[
\text{either }B_{k,\eta }=0\quad\text{or}\quad\eta (p)p^{k}=1\text{ for some prime $p$ dividing~$N$}.
\]
\end{theorem}

If the representation $\rho$ alluded to above is not strongly modular, we give in Theorem~\ref{main2} below a precise characterization (under the same assumption as before) of the primes $M\nmid Nl$ for which $\rho$ arises from a cusp form of type~$(NM,k,\eps)$. Such a theorem extends a result of Mazur (\cite[Prop.~5.12]{Maz77}), that handles  the case $(N,k)=(1,2)$, to arbitrary weights and conductors.
\begin{theorem}\label{main2}
In the same notation and under the same assumptions as in Theorem~\ref{main}, assume moreover that $\rho$ is not strongly modular. Let $M$ be a prime number not dividing~$Nl$. Then~$\rho$ arises from a modular form of type~$(NM,k,\varepsilon)$ if and only if 
\[
\left\{
\begin{array}{ll}
M\equiv1\pmod{l} & \text{if }(N,k)=(1,2) \quad \text{(Mazur)}\\
\eta(M)M^k=1 & \text{if }(N,k)\not=(1,2).
\end{array}
\right.
\]
In particular, there are infinitely many such primes.
\end{theorem}

We remark that, due to the results of Carayol already mentioned, the modular form over $\Ffbar$ in Theorem \ref{main2} can be taken to be the reduction of a newform  of level $NM$ (\emph{cf}. subsection \ref{conducteur} of this article). 

Although the details need to be treated separately, the overall strategy for proving both results is the same and relies on properties of characteristic zero eigenforms and their attached automorphic representations. Let us briefly describe this strategy in the case of Theorem~\ref{main}. Let~$\rho$ be as the statement of the theorem. Attached to such a reducible representation is a specific Eisenstein series~$E$. If $\rho$ is strongly modular, then there must occur a congruence between~$E$ and a certain cuspidal (new) eigenform of weight~$k$ and level~$N$. This in turn implies that the constant terms of~$E$ vanish at all cusps after reduction modulo~$l$, leading to the necessary conditions of the theorem. Conversely, if these conditions hold, then we prove that the reduction of~$E$ modulo~$l$ is a cusp form~$f$ over~$\Ffbar$ of the same type as~$\rho$ such that~$\rho\simeq\rho_f$.

The paper is organized as follows. In Section~\ref{s:Bernoulli}, we define the Bernoulli numbers attached to mod~$l$ Galois characters that appear in the statement of Theorem~\ref{main} above and compute the constant term at the various cusps of a particular Eisenstein series which is of crucial use in the proofs of our results. After quickly recalling in Section~\ref{s:background} some background on cuspidal eigenforms and Hecke operators in the adelic setting, we prove in Section~\ref{s:proofs} our two main theorems. 

\medskip

\noindent {\bf Acknowledgements:} The authors wish to thank Vinayak Vatsal for ins\-piring discussions and the Pacific Institute for the Mathematical Sciences in Vancouver for providing ideal conditions to carry out part of this project. We also thank the anonymous referee for precise comments that have improved the exposition.

\section{Bernoulli numbers and Eisenstein series}\label{s:Bernoulli}

In this section we recall some classical definitions and integrality results on Bernoulli numbers attached to Dirichlet characters. Also, we compute the constant term in the $q$-expansion at the  cusps of the modular curve $X_1(N)$ of some specific Eisenstein series that will be used in the sequel. The final computation is stated in Proposition \ref{prop:cst_terms} below. 

\subsection{Notation and definitions}\label{ss:notation}
Let~$\phi$ be a primitive Dirichlet character of conductor~$\mathfrak{f}\ge 1$. The Gauss sum attached to~$\phi$ is defined by
\begin{equation*}
W(\phi)=\sum_{n=1}^{\mathfrak{f}}\phi(n)e^{2i\pi n/\mathfrak{f}}.
\end{equation*}
It is a non-zero algebraic integer whose norm is a power of~$\mathfrak{f}$. The (gene\-ralized) Bernoulli numbers $(B_{m,\phi})_{m\ge1}$ associated with~$\phi$ are defined by the following expansion
\begin{equation}\label{eq:def_gene_Bernoulli}
\sum_{n=1}^{\mathfrak{f}}\phi(n)\frac{te^{nt}}{e^{\mathfrak{f}t}-1}=\sum_{m\ge0}B_{m,\phi}\frac{t^m}{m!}.
\end{equation}
Note that when $\phi=\trivial$ is the trivial character (of conductor~$1$), then, for every integer~$m\ge2$, we have $B_{m,\phi}=B_m$ where $B_m$ denotes the classical $m$-th Bernoulli number.

\subsection{Bernoulli numbers of mod~$l$ characters}\label{ss:Bernoulli_mod_l} Let $\eta \colon\GQ\rightarrow\Ffbar^{\times}$ be a Galois character unramified at~$l$. Denote by $\cond{0}$ the conductor of~$\eta$ (coprime to~$l$ by assumption) and identify~$\eta$ with a character
\[
\eta\colon(\Z/\cond{0}\Z)^{\times}\longrightarrow\Ffbar^{\times}.
\]
The aim of this paragraph is to define the $k$-th Bernoulli number attached to~$\eta $ for any integer $k\geq2$ such that~\(l>k+1\).
This definition relies on integrality properties of Bernoulli numbers attached to Dirichlet characters which we now recall.

Let~\(\place\) be a place of~\(\Qbar\) above~\(l\) and let~$\Zbar_{\place}$ be the local ring of $w$-integral algebraic numbers in~$\Qbar$. The residue field~$k_{\place}$ of~$\place$ identifies with an algebraic closure of~$\Ff$. Fix an isomorphism $\iota\colon k_{\place}\rightarrow\Ffbar$ and consider the composition map
\[
\nu_{\place}\colon\Zbar_{\place}\rightarrow k_{\place}\stackrel{\iota}{\rightarrow}\Ffbar.
\]
We may then consider the multiplicative lift
\[
\psi\colon(\Z/\cond{0}\Z)^{\times}\longrightarrow\Zbar^{\times}
\]
of~\(\eta \) with respect to~\(\place\). That is, $\psi$ is the unique character with values in the roots of unity of prime-to-$l$ order such that
\[
\nu_\place(\psi(x))=\eta (x),\quad\text{for all }x\in(\Z/\cond{0}\Z)^{\times}.
\]

We now state the integrality result we need to define our Bernoulli numbers associated to $\eta$.
\begin{lemma}\label{lem:Carlitz}
For any integer~$k\geq2$ such that $l>k+1$, the Bernoulli number~$B_{k,\psi}$ is $w$-integral.
\end{lemma}
\begin{proof} Let~$k$ be an integer as in the statement of the lemma. We easily check on the definition~(\ref{eq:def_gene_Bernoulli}) that if $\psi(-1)\not=(-1)^k$, then $B_{k,\psi}=0$. Assume therefore that~$\psi(-1)=(-1)^k$.  If $\psi$ is the trivial character, then $k$ must be an even integer and the corresponding Bernoulli number $B_{k,\psi}$ is nothing but the classical Bernoulli number~$B_k$. The Van Staudt-Clausen theorem ensures that the prime divisors $p$ of the denominator of $B_k$ satisfy $p-1\mid k$. Since $l>k+1$, the prime number $l$ does not divide the denominator of $B_k$, as desired.

Assume therefore $\psi$ is non trivial. Let $$d=\left\{ 
\begin{array}{ll}
 1 & \textrm{ if } \cond{0} \textrm{ admits two different prime divisors} \\
 2 & \textrm{ if } \cond{0}=4 \\
 1 & \textrm{ if } \cond{0}=2^n, n>2\\
 k\cond{0}  & \textrm{ if } \cond{0}>2 \textrm{ is a prime number}\\
 1-\psi(1+p) & \textrm{ if } \cond{0}=p^n, p>2, n>1, p \textrm{ is a prime number. } 
\end{array}\right.$$

By a theorem of Carlitz (see \cite{Car59a} and~\cite{Car59b}), $dk^{-1}B_{k,\psi}$ is an algebraic integer. Hence, we are reduced to verify that~$w$ does not divide~$d$. 

Assume that $\cond{0}=p^n$, where $p$ is an odd prime number and $n \geq 2$. We assume by contradiction that $w$ divides $d=1-\psi(1+p)$. Let $H\subseteq (\Z/p^n\Z)^\times$ be the subgroup  spanned by~$1+p$.  Taking the reduction map $\nu_\place$ attached to~$\place$, we conclude that  $\eta $ is trivial on $H$.  Since  $H$ is the kernel of the natural map $(\Z/p^n\Z)^\times \rightarrow (\Z/p\Z)^\times$, we conclude that $\eta $ can be factored through $(\Z/p\Z)^\times$, contradicting the primitivity of $\eta$. 

If $\cond{0}\geq 3$  is not of the form discussed in the previous paragraph, the fact that $l \nmid d$ clearly follows from the definition of $d$ and the hypothesis on $k,l$ and $\cond{0}$.
\end{proof}

Using this result, we now set, for any integer~$k$ as above,
\begin{equation}\label{eq:def_Bernoulli}
B_{k,\eta }=\nu_{\place}\left(B_{k,\psi}\right)\in\Ffbar.
\end{equation}
Let $\place'$ be another place of~$\Qbar$ over~$l$. There exists $\sigma\in\GQ$ such that $\place'=\sigma(\place)$ and we identify the residue field $k_{\place'}$ (of the ring of $\place'$-integral algebraic numbers in~$\Qbar$) with~$\Ffbar$ via~$\iota\circ\sigma^{-1}$. Then~$\sigma(\psi)$ is the multiplicative lift of~$\eta$ with respect to the place~$\place'=\sigma(\place)$ and since we have
\[
B_{k,\sigma(\psi)}=\sigma\left(B_{k,\psi}\right)
\]
the definition~(\ref{eq:def_Bernoulli}) is independent of the choice of the place~$\place$. We refer to $B_{k,\eta }\in\Ffbar$ as the $k$-th Bernoulli number associated with~$\eta $.

\subsection{The setting}\label{ss:setting}
In this paragraph we set some notation and definitions that will be used in the rest of this section. Let $k\ge2$ be an integer. We set
$$
C_k=\frac{(-2i\pi)^k}{(k-1)!}.
$$
Let
\[
\chi_i\colon(\Z/\cond{i}\Z)^{\times}\longrightarrow\C^{\times},\quad i=1,2
\]
be primitive Dirichlet characters such that $\chi_1(-1)\chi_2(-1)=(-1)^k$. Denote by $\overline{\chi_i}$ the complex conjugate of~$\chi_i$, $i=1,2$. Put $N=\cond{1}\cond{2}$. For $k\geq 3$ and $z$ in the complex upper-half plane~$\h$, let  
$$
G_{k}^{\chi_1,\chi_2}(z)=\sum_{(m,n)\in \Z^2\backslash \{(0,0)\} }^{} \frac{\chi_1(m)\overline{\chi_2}(n)}{(mz+n)^k}.
$$
On the other hand, for any $\eps>0$, we consider
$$
G_{2, \eps}^{\chi_1,\chi_2}(z)=\sum_{(m,n)\in \Z^2\backslash \{(0,0)\} }^{} \frac{\chi_1(m)\overline{\chi_2}(n)}{(mz+n)^2|mz+n|^{2\eps}},\quad (z\in\h).
$$
We remark that our functions $G_k^{\chi_1,\chi_2}(z)$ ($k\ge3$) and $G_{2, \eps}^{\chi_1,\chi_2}(z)$ correspond to the functions $E_k(z;\chi,\psi)$ and $E_2(z,\eps;\chi,\psi)$ respectively defined in Eq.~ (7.1.1) and~(7.2.1) of~\cite{Miy06} with~$(\chi,\psi)=(\chi_1,\overline{\chi_2})$. 

From now on, and until the end of this section, assume that either $N>1$ or $k>2$ and denote by $E_k^{\chi_1,\chi_2}$ the function defined by 
\begin{equation}\label{eq:ES}
E_k^{\chi_1,\chi_2}(z)=
-\delta(\chi_1)\frac{B_{k,\chi_2}}{2k}
+\sum_{n\ge1}\sigma_{k-1}^{\chi_1,\chi_2}(n)q^n, \quad (q=e^{2\pi i z}, z\in\h)
\end{equation}
where 
\[\sigma_{k-1}^{\chi_1,\chi_2}(n)=
\sum_{m\mid n}\chi_1(n/m)\chi_2(m)m^{k-1},
\quad
\delta(\chi_1)=\left\{
\begin{array}{ll}
1 & \text{if \(\chi_1\) is trivial} \\
0 & \text{otherwise}
\end{array}
\right.
\]
and~\(B_{k,\chi_2}\) denotes the \(k\)-th Bernoulli number associated with~\(\chi_2\) (see paragraph~\ref{ss:notation}).

According to~\cite{Miy06}, Thm. 7.1.3 and Eq.~(7.1.13), we have
\begin{equation}\label{peso superior}
G_{k}^{\chi_1,\chi_2}(\cond{2}z) =\frac{2C_kW(\overline{\chi_2})}{\cond{2}^k }  E_{k}^{\chi_1,\chi_2}(z), \quad\text{for }k \geq 3,
\end{equation}
and similarly using Thm. 7.2.12
\begin{equation}\label{peso2}
\lim_{\eps \rightarrow 0^+} G_{2, \eps}^{\chi_1,\chi_2}(\cond{2}z) = \frac{2C_2 W(\overline{\chi_2})}{\cond{2}^2}  E_2^{\chi_1, \chi_2}(z).  
\end{equation}

According to \emph{loc. cit.} \S7.1 and \S7.2 for $k\ge3$ and $k=2$ respectively, together with Thm.~4.7.1, we have that $E_k^{\chi_1,\chi_2}$ is an Eisenstein series of weight~$k$, level~$N$ and Nebentypus character~$\chi_1\chi_2$.

\subsection{Computation of the constant terms}

We keep the notation and assumptions of the previous paragraph and moreover denote by~\(\cond{0}\)  the conductor of the primitive character~\((\overline{\chi_1}\chi_2)_0\) associated with~$\overline{\chi_1}\chi_2$. For any integer $M$ we denote by $\alpha_M$ the usual degeneracy operator given by $\alpha_M f(z)=f(Mz).$

For a given matrix \(\gamma\in\SL_2(\Z)\), we let
$$
\Upsilon_k^{\chi_1,\chi_2}(\gamma,M)=\lim_{\im(z)\rightarrow \infty} \Big(\left(\alpha_ME_k^{\chi_1,\chi_2}\right)|_k \gamma\Big)(z)
$$  
be the constant term of the Fourier expansion at~$\infty$ of~$\left(\alpha_ME_k^{\chi_1,\chi_2}\right)|_k\gamma$. Here, the notation $\vert_k$ refers to the classical slash operator acting  on weight $k$ modular forms.

The main goal of this section is the computation, embodied in Proposition~\ref{prop:cst_terms} below, of the constant term $\Upsilon_k^{\chi_1,\chi_2}(\gamma,M)$. 
 
\begin{prop}\label{prop:cst_terms}
Let \(\gamma=\begin{pmatrix}
             u & \beta \\
	     v & \delta
             \end{pmatrix}\in\SL_2(\Z)\) and let~\(M\ge1\) be an integer. Put  $r=\gcd(v,M)$, \(v'=v/r\) and $M'=M/r$.  If~\(\cond{2}\nmid v'\), then we have that $\Upsilon_k^{\chi_1,\chi_2}(\gamma,M)=0$. Else, if \(\cond{2}\mid v'\), then
\[
\Upsilon_k^{\chi_1,\chi_2}(\gamma,M)\not=0\Longleftrightarrow \gcd\left(\frac{v'}{\cond{2}},\cond{1}\right)=1.
\]
Moreover, in that case, we have that $\Upsilon_k^{\chi_1,\chi_2}(\gamma,M)$ is given by the following non-zero algebraic number 
\[
\Upsilon_k^{\chi_1,\chi_2}(\gamma,M)=\xi \cdot 
\left(\frac{\cond{2}}{M'\cond{0}}\right)^k\cdot 
\frac{W\left((\chi_1\overline{\chi_2})_0\right)}{W\left(\overline{\chi_2}\right)}\cdot
\frac{B_{k,(\overline{\chi_1}\chi_2)_0}}{2k}
\prod_{p\mid N}\left(1-\left(\chi_1\overline{\chi_2}\right)_0(p)p^{-k}\right)
\]
where $\xi=-\chi_2(\delta)\overline{\chi_2}(M')\chi_1\left(-v'/\cond{2}\right)$ is a root of unity and $p$ runs over the prime divisors of~$N$.
\end{prop}

\begin{rk}
The result above generalizes the special cases $(\chi_1,\chi_2,M)=(\chi_1,\chi_1^{-1},1)$ and $(\chi_1,k)=(\trivial,\ge3)$  stated in~\cite[Prop.~2.8]{BiDi14} and~\cite[Prop.~1.2]{BiMe15} respectively. In this paper, we not only need the above statement in its full generality and precision, but we also provide a unified and (slightly) simplified proof of these previous results.
\end{rk}

The following result is easily deduced from the above proposition and will be of use in Section~\ref{s:proofs}.
\begin{cor}\label{cor:cst_terms}
In the notation of Proposition~\ref{prop:cst_terms}, assume~$M$ and $N$ are coprime. Then, we have
\[
\Upsilon_k^{\chi_1,\chi_2}(\gamma,M)
=\left(\frac{r}{M}\right)^k\overline{\chi_1}(r)\chi_2(r)\overline{\chi_2}(M)\Upsilon_k^{\chi_1,\chi_2}(\gamma,1).
\]
\end{cor}

We break the proof of Proposition \ref{prop:cst_terms} in several steps.  The proof is given at the end of this paragraph, except for the justification of an intermediary step in the case $k=2$, which is dealt with in the next paragraph.

\begin{lemma}\label{sumsgeneral} Under the same hypothesis as in Proposition \ref{prop:cst_terms}, we have that 
$$\Upsilon_k^{\chi_1,\chi_2}(\gamma,M)= \frac{\cond{2}^k}{2C_kW(\overline{\chi_2})}\cdot \sum_{(m,n) \in C }^{ } \frac{\chi_1(m)\overline{\chi_2}(n)}{(mM\cond{2}\beta + n\delta)^k},$$
where $C=\left\{(m,n)\in \Z^2\backslash \{(0,0)\} : mM\cond{2}u+nv=0\right\}.$
\end{lemma}
\begin{proof}[Proof of Lemma \ref{sumsgeneral} in the case $k >2$]
Using~\eqref{peso superior}, we have that
\begin{equation*}
 \frac{2C_kW(\overline{\chi_2})}{\cond{2}^k}  \Upsilon_k^{\chi_1,\chi_2}(\gamma,M)=\lim_{\im(z) \rightarrow \infty}\Big(\left(\alpha_{M\cond{2}}G_{k}^{\chi_1,\chi_2}\right)|_k\gamma\Big)(z).
\end{equation*}

Besides, we have
\begin{equation*}
\Big( \left(\alpha_{M\cond{2}}G_{k}^{\chi_1,\chi_2}\right)\vert_k\gamma\Big)(z) =  \sum_{\substack{(m,n)\in\Z^2 \\ (m,n)\neq (0,0) }} \frac{\chi_1(m)\overline{\chi_2}(n)}{( z(mM\cond{2}u+nv)+mM\cond{2}\beta+n\delta  )^{k}} 
\end{equation*}
where the above sum is absolutely convergent since $k \geq 3$. We can therefore exchange limit and summation, yielding the result.
\end{proof}

\begin{rk}
 When $k=2$, the sum in the last equation of the previous proof  is not absolutely convergent and it becomes necessary to give additional considerations, that we present in paragraph \ref{parrafotecnico}, in order to justify the interchange of limit and summation. The full proof of Lemma \ref{sumsgeneral} is thus achieved in Lemma~\ref{sums} below.
\end{rk}

We now prove the following key result assuming the validity of Lemma~\ref{sumsgeneral} for any~$k\ge2$.
\begin{lemma}\label{valeurs}
Under the same hypothesis as in Proposition \ref{prop:cst_terms}. If~\(\cond{2}\nmid v'\), then  we have $\Upsilon_k^{\chi_1,\chi_2}(\gamma,M)=0$. Else, if \(\cond{2}\mid v'\), then we have

$$
\Upsilon_k^{\chi_1,\chi_2}(\gamma,M)=\frac{\overline{\chi_2}(M'u)\chi_1(-v'/\cond{2})}{M'^k}\cdot \frac{\cond{2}^k}{C_kW(\overline{\chi_2})} \cdot L(k,\chi_1\overline{\chi_2}),
$$   
where $L(k,\chi_1\overline{\chi_2})=\sum_{n\geq1}(\chi_1\overline{\chi_2})(n)n^{-k}$.
\end{lemma}
\begin{proof}
For simplicity, put $\Upsilon= \Upsilon_k^{\chi_1,\chi_2}(\gamma,M)$. 

\begin{enumerate}[(i)]
 \item  Assume $u=0$. Then, $-v\beta=  1,$ implying $v \in \{\pm 1\}$,  $M'=M$ and $v'=v$. Also, the set $C$ in Lemma~\ref{sumsgeneral}  satisfies $C = (\Z\setminus\{0\}) \times \{0\}$. If $\chi_2 \neq \trivial$ (that is, if $\cond{2}\nmid v'$), we have that $\chi_2(0)=0$ and then $\Upsilon=0$ as claimed.

Assume now that $\chi_2=\trivial$. Then, $\cond{2}=W(\overline{\chi_2})=1$ and $ \chi_1(-1)=(-1)^k$. These relations imply $\chi_1(-v)=\beta^{-k}$. On the other hand,   Lemma~\ref{sumsgeneral} ensures that 
$$2C_k \Upsilon = \sum_{\substack{m \in \Z \\ m \neq 0 }}^{ } \frac{\chi_1(m)}{(mM\beta)^k} = \frac{2}{(\beta M)^k}L(k,\chi_1)=\frac{2\chi_1(-v)}{M^k}L(k,\chi_1),$$ concluding the proof in this case.

\item Assume $u \neq 0$. We have the following

\begin{claim*}
 Let $n \in \Z\setminus\{0\}$ with $\gcd(n,\cond{2})=1$. Then,  there exists $m\in \Z $ such that $(m,n) \in C$ if and only if $M'u|n$ and $\cond{2}\mid v'$. Furthermore, in this case we have that 
\begin{equation}\label{clair}
m=-\frac{n}{M'u}\cdot\frac{v'}{\cond{2}}\quad\text{and}\quad mM\cond{2}\beta +n \delta=\frac{n}{u}.
\end{equation}
\end{claim*}
\begin{proof}[Proof of the claim]
If $(m,n) \in C$, then  $mM'\cond{2}u+nv'=0$. We have that $\gcd(M',v')=1$  by definition. Moreover, $\gamma \in\SL_2(\Z)$ implies $\gcd(u,v)=1$, hence $M'u\mid n$.   On the other hand, since $\gcd(\cond{2}, n)=1$, we have that $\cond{2}\mid v'$. 

Conversely, if $M'u\mid n$ and $\cond{2}\mid v'$, then the integer $m=-\frac{n}{M'u}\cdot \frac{v'}{\cond{2}}$ satisfies $(m,n) \in C$. 

Finally, if the equivalence is satisfied, we easily check using the relation $u\delta-v\beta=1$, that the second relation in Eq.~\eqref{clair} holds.
\end{proof}
 
 Using the claim and Lemma \ref{sumsgeneral}, we have that $\Upsilon=0$ if $\cond{2}\nmid v'$. Else, if $\cond{2}\mid v'$, then we have

\begin{eqnarray*}
\frac{2C_kW(\overline{\chi_2})}{\cond{2}^k}  \Upsilon &=& \sum_{\substack{M'u|n \\ n \neq 0}}^{ } \frac{\chi_1(n/M'u)\chi_1(-v'/\cond{2})\overline{\chi_2}(n) }{\left(\frac{n}{u}\right)^k } \\
 &=&  \chi_1(-v'/\cond{2})  \sum_{\substack{t \in \Z \\ t \neq 0}}^{ } \frac{\chi_1(t)\overline{\chi_2}(M'ut) }{(M't)^k } \quad (n=M'ut) \\
&=& \frac{ \chi_1(-v'/\cond{2}) \overline{\chi_2}(M'u)}{M'^k} \sum_{\substack{t \in \Z \\ t \neq 0}}^{ } \frac{\chi_1(t)\overline{\chi_2}(t) }{t^k}  \\
&=& \frac{ \chi_1(-v'/\cond{2})\overline{\chi_2}(M'u)}{M'^k} 2 L(k,\chi_1\overline{\chi_2}),
\end{eqnarray*}
since $\chi_1(-1)\overline{\chi_2}(-1)=(-1)^k$. This finishes the proof of Lemma~\ref{valeurs}.
\end{enumerate}
\end{proof}

\begin{proof}[Proof of Proposition~\ref{prop:cst_terms}]  According to Lemma~\ref{valeurs}, it remains to deal with the case where $\cond{2}\mid v'$. In that case, by reducing the equality $u\delta-v\beta=1$ modulo~$\cond{2}$, we get $u\delta\equiv1\pmod{\cond{2}}$. Besides, we have $\gcd(M',\cond{2})\mid v'$ and hence ${\gcd(M',\cond{2})=1}$. Therefore if we assume that $\gcd(v'/\cond{2},\cond{1})=1$, it follows that
$$
-\chi_1(-v'/\cond{2})\overline{\chi_2}(M'u)=-\chi_1(-v'/\cond{2})\chi_2(\delta)\overline{\chi_2}(M')=\xi
$$  
is a root of unity.

Besides, by~\cite[(3.3.14)]{Miy06}, we have
$$L(k,\chi_1\overline{\chi_2})=L(k,(\chi_1\overline{\chi_2})_0)\prod_{p|N} \left( 1-\frac{(\chi_1\overline{\chi_2})_0(p)}{p^k}\right),$$
where $(\chi_1\overline{\chi_2})_0$ denotes the primitive character associated with~$\chi_1\overline{\chi_2}$. Moreover, it follows from the Euler product for $L(k,(\chi_1\overline{\chi_2})_0)$ that $L(k,\chi_1\overline{\chi_2})\not=0$.

Now using the assumption $(\chi_1\overline{\chi_2})_0(-1)=(-1)^k$ and \cite[Thm.~3.3.4]{Miy06}, we get that 
$$
L(k,(\chi_1\overline{\chi_2})_0)=-W((\chi_1\overline{\chi_2})_0)\cdot\frac{C_k}{\cond{0}^k}\cdot\frac{B_{k,(\overline{\chi_1}\chi_2)_0}}{2k}.
$$
Combining these facts together with  Lemma \ref{valeurs} concludes the proof of Proposition~\ref{prop:cst_terms}.
\end{proof}

\subsection{The case of weight 2}\label{parrafotecnico}

The goal of this paragraph is to prove Lemma~\ref{sumsgeneral} in the case $k=2$. This is achieved in Lemma \ref{sums}. For $\eps\geq0$, we use the notation 
$$w^{2,\eps}=w^2|w|^{2\eps}, \quad w \in \C.$$

Let $y_0>0$ be a positive real number. The notation  $g_1\ll_{y_0} g_2$ means that there exists a positive constant $C$, depending only on $y_0$, such that $|g_1(r)|\leq C |g_2(r)|$ for all $r$ in the common domain of $g_1,g_2$.

Let
$$S_\eps(z)=\sum_{n \in \Z }^{ } \frac{1}{(z+n)^{2,\eps}},\quad z\in\C\setminus\R.$$
For~$z\in\h$, the function $S_\eps(z)$ corresponds to the function~$S(z;2+\eps,\eps)$ in the notation of~\cite[(7.2.7)]{Miy06}.
\begin{lemma}\label{mastecnico}
Fix $y_0>0$. Then, we have that 
$$S_\eps(z) \ll_{y_0}\frac{1}{\Gamma(\eps)|y|^{1+2\eps}}+e^{-2\pi |y|}, \quad y=\im(z), \quad |y|\geq y_0, \quad 0<\eps\leq1,$$
where for any real number $s>0$, $\Gamma(s)=\int_0^{\infty}e^{-t}t^{s-1}dt$.
\end{lemma}

\begin{proof} Since we have  $S_\eps(x-iy)=S_\eps(-x+iy)$, we can assume that $y\geq y_0$. For $m\in\Z$, let us denote by $\xi_\eps(y;m)$ the function~$\xi(y;2+\eps,\eps;m)$ of~\cite[(7.2.11)]{Miy06}. According to Theorem~7.2.8 of~\emph{loc. cit.}, we then have 

\begin{equation}\label{Fourier}
 S_\eps(z)=\xi_\eps(y;0)+\sum_{\substack{m \in \Z \\ m \neq 0} }^{ } e^{2\pi i mx} \xi_\eps(y;m), \quad z=x+iy,
\end{equation}
where the series converges absolutely. Besides, for $m\in\Z$, we have by \emph{loc. cit.}, Theorem~7.2.5, that
$$
\xi_\eps(y;m)= \left\{ 
\begin{array} {ll}
-\frac{(2\pi)^{2+\eps}}{\Gamma(2+\eps)} \frac{1}{(2y)^\eps} m^{1+\eps}e^{-2\pi ym}\omega(4\pi ym;2+\eps,\eps) & \textrm{ if } m>0 \\
-\frac{(2\pi)^{2+2\eps}\Gamma(1+2\eps)}{\Gamma(2+\eps)\Gamma(\eps)} \frac{1}{(4\pi y)^{1+2\eps}} &  \textrm{ if } m=0 \\
-\frac{(2\pi)^{\eps}}{\Gamma(\eps)} \frac{1}{(2y)^{2+\eps}} \frac{1}{|m|^{1-\eps}} e^{-2\pi y|m|}\omega(4\pi y|m|;\eps,2+\eps) & \textrm{ if } m<0.
\end{array}\right.
$$
The definition of the function $\omega$ is stated in \emph{loc. cit.} (7.2.31). It follows from Theorem 7.2.7 in~\emph{loc. cit.} that for all $m\in\Z\setminus\{0\}$, $y\geq y_0$ and $0<\eps\leq 1$, we have
$$\omega(4\pi y |m|;2+\eps,\eps) \ll_{y_0}1 \quad\text{and}\quad \omega(4\pi y |m|;\eps,2+\eps) \ll_{y_0}1.$$

Therefore, for all $y\geq y_0$ and  $0<\eps\leq 1$ we have
$$
\xi_\eps(y;m)\ll_{y_0} \left\{ 
\begin{array} {ll}
m^2e^{-2\pi ym} & \textrm{ if } m>0 \\
\frac{1}{\Gamma(\eps)y^{1+2\eps}} &  \textrm{ if } m=0 \\
e^{-2\pi y|m|} & \textrm{ if } m<0
\end{array}\right.
$$
and Eq. \eqref{Fourier} implies 

$$S_\eps(z) \ll_{y_0} \frac{1}{\Gamma(\eps)y^{1+2\eps}} +\sum_{m \geq 1 }  m^2 e^{-2\pi y m} + \sum_{m \geq 1 }^{ } e^{-2\pi y m}  .$$

On the other hand, for all $y\geq y_0$, we have
$$\sum_{m \geq 1 }^{ } (m^2+1)e^{-2\pi y m}= \frac{e^{-2\pi y}(e^{-4\pi y}-e^{-2\pi y}+2)}{(1-e^{-2\pi y})^3}\ll_{y_0} e^{-2\pi y},$$ 
hence the result follows.
\end{proof}

\begin{lemma}\label{tecnico}
For any $a_1,a_2,D \in \Z$ with $D \neq 0$, set
$$\sigma_\eps(z;a_1,a_2,D)=\sum_{\substack{(m,n)\in \Z^2 \\ a_1+Dm \neq 0  }} \frac{1}{(z(a_1+Dm)+a_2+Dn)^{2,\eps}}.$$
Then, we have that
$$\lim_{\im(z)\rightarrow \infty} \lim_{\eps \rightarrow 0^+} \sigma_\eps(z;a_1,a_2,D) =0.$$
 
\end{lemma}

\begin{proof} Assume $y=\im(z)\geq 1$. We have that

\begin{align*}
\sigma_\eps(z;a_1,a_2,D)&=  \frac{1}{D^{2,\eps}} \sum_{\substack{m\in \Z \\ a_1+Dm \neq 0} } \sum_{n \in \Z }^{ } \frac{1}{\left(z(\frac{a_1}{D} +m) +\frac{a_2}{D}+n  \right)^{2,\eps}}     \\
 &= \frac{1}{D^{2,\eps}} \sum_{\substack{m \in \Z \\  a_1+Dm \neq 0} }^{ } S_\eps \left( z\left(\frac{a_1}{D} +m\right) +\frac{a_2}{D}  \right).
\end{align*}
Define
\[
y_0=\min\left\{\left|\im\bigg(z\left(\frac{a_1}{D} +m\right) \bigg)\right|; \im(z)\geq 1, m\in\Z : a_1 +Dm\neq 0\right\}.
\]
Since $\Z$ is discrete, we have $y_0>0$. Using Lemma \ref{mastecnico} with this choice of $y_0$, we find that for $\eps\leq 1\leq y$ and $m\in\Z$ such that $a_1 +Dm\neq 0$, we have
\[
 S_\eps \left( z\left(\frac{a_1}{D} +m\right) +\frac{a_2}{D}  \right)\ll_{y_0} \frac{1}{\Gamma(\eps)y^{1+2\eps}} \cdot  \frac{1}{\left| \frac{a_1}{D} +m\right|^{1+2\eps} } +e^{-2\pi y \left|\frac{a_1}{D}+m\right| }.
\]
Therefore, we have
\begin{equation*}
\sigma_\eps(z;a_1,a_2,D) \ll_{y_0} \frac{1}{|D|^{2(1+\eps)}} \left( \frac{1}{y^{1+2\eps}}\cdot \frac{1}{\Gamma(\eps)}   \cdot \zeta(1+2\eps)+  \sum_{n\geq 1 }^{ } e^{-\frac{2\pi yn}{|D|} }    \right).
\end{equation*}

Since $\sum_{n\geq 1 }^{ } e^{-\frac{2\pi yn}{|D|} } \ll_{y_0} e^{-\frac{2\pi y}{|D|} }$, we have that 
$$
\limsup_{\eps\rightarrow 0^+}\left|\sigma_\eps(z;a_1,a_2,D)\right| \ll_{y_0} \frac{1}{D^2} \left(\frac{1}{y}+e^{-\frac{2\pi y}{|D|} }\right).
$$
This estimate justifies the claim.
\end{proof}

\begin{lemma}\label{sums} 
Lemma \ref{sumsgeneral} is true for $k=2$. 
\end{lemma}

\begin{proof} Using Eq.~\eqref{peso2}, we have in particular that
\begin{equation}\label{terminoconstante}
\Upsilon_2^{\chi_1,\chi_2}(\gamma,M)= \frac{\cond{2}^2}{2C_2W(\overline{\chi_2})} \lim_{\im(z) \rightarrow \infty} \lim_{\eps\rightarrow 0^+}\Big( \left(\alpha_{M\cond{2}}G_{2,\eps}^{\chi_1,\chi_2}\right)|_2\gamma\Big)(z).
\end{equation}
For $\eps>0$, let
\begin{equation}\label{T}
T_{\eps}(z) = \sum_{(m,n) \in C }  \frac{\chi_1(m)\overline{\chi_2}(n)}{( mM\cond{2}\beta+n\delta  )^{2,\eps}} 
\end{equation}
and
$$
 R_{\eps}(z) =\sum_{\substack{(m,n) \notin C \\ (m,n) \neq (0,0)}}  \frac{\chi_1(m)\overline{\chi_2}(n)}{( z(mM\cond{2}u+nv)+mM\cond{2}\beta+n\delta  )^{2,\eps}},
$$
where, as in Lemma~\ref{sumsgeneral}, 
$$C=\left\{(m,n)\in \Z^2\backslash \{(0,0)\} : mM\cond{2}u+nv=0\right\}.$$ 
Then, we have 
\begin{equation*}
\left(\alpha_{M\cond{2}}G_{2,\eps}^{\chi_1,\chi_2}\right)\vert_2\gamma(z) = |vz+\delta|^\eps \sum_{\substack{(m,n)\in\Z^2 \\ (m,n)\neq (0,0)} }^{ } \frac{\chi_1(m)\overline{\chi_2}(n)}{(mM\cond{2}(uz+\beta) +n(vz+\delta))^{2,\eps}}
\end{equation*}
and therefore 
$$
\lim_{\eps \rightarrow 0^+}\left(\alpha_{M\cond{2}}G_{2,\eps}^{\chi_1,\chi_2}\right)\vert_2\gamma(z)= \lim_{\eps\rightarrow 0^+} \left(T_\eps(z) + R_\eps(z)\right).
$$
Since the parameters appearing in the sum defining $T_\eps$ are linked by a linear relation, the series obtained by setting $\eps=0$ in \eqref{T} is absolutely convergent. Hence, we have that

$$\lim_{\eps \rightarrow 0^+} T_\eps(z)=  \sum_{(m,n) \in C }  \frac{\chi_1(m)\overline{\chi_2}(n)}{( mM\cond{2}\beta+n\delta  )^{2}}.$$
In particular, this limit is independent of $z$. Hence, in light of Eq.~\eqref{terminoconstante}, in order to finish the proof we need to show that 

\begin{equation}\label{final}
\lim_{y \rightarrow \infty} \lim_{\eps\rightarrow 0^+} R_\eps(z)=0.
\end{equation}
We have that
\begin{equation}\label{erre}
 R_{\eps}(z) = \sum_{a=0}^{\cond{1}-1}\sum_{b=0}^{\cond{2}-1}\chi_1(a)\overline{\chi_2}(b) \sum_{\substack{(c,d)\in C_{a,b}\\c \neq 0 }}^{ }\frac{1}{( cz+d )^{2,\eps}},
\end{equation}
where $$C_{a,b}=\left\{ (mM\cond{2}u+nv,mM\cond{2}\beta+n\delta ): m \equiv a \mod \cond{1}, n \equiv b \mod \cond{2}\right\}.$$

Now we proceed to split each of the sums in \eqref{erre} indexed by $C_{a,b}$ in a finite number of sums of the type handled by Lemma~\ref{tecnico}. Let $$\mathbb{M}=\left(\begin{array}{cc}
M\cond{1}\cond{2}u & M\cond{1}\cond{2}\beta \\
\cond{2}v & \cond{2}\delta
\end{array}\right), \quad \theta^{a,b}=\left(aM\cond{2}u+bv,aM\cond{2}\beta +b\delta \right).$$  
Then, $C_{a,b} = \theta^{a,b} +\Z^2\cdot\mathbb{M}$ (here, we represent the elements of $\Z^2$ as row vectors). Let  $D:=\det \mathbb{M}=M\cond{1}\cond{2}^2$. By the elementary divisors theorem, we have that 
$ D\Z \times D\Z \subset\Z^2\cdot\mathbb{M} $ is a subgroup of index~$D$. Let $\{ r_1,r_2,\ldots,r_D\}$ be a system of representatives of the quotient $\Z^2\cdot\mathbb{M}/D\Z \times D\Z$. Then, in the notation of Lemma~\ref{tecnico}, we have that

$$
R_{\eps}(z) =\sum_{a=0}^{\cond{1}-1}\sum_{b=0}^{\cond{2}-1} \chi_1(a)\overline{\chi_2}(b) \sum_{i=1}^D  \sigma_\eps\left(z;\theta_1^{a,b}+r_{i,1},\theta_2^{a,b}+r_{i,2},D\right),
$$
where, for any vector $w\in \R^2$ we write $w=(w_1,w_2)$. Then, using Lemma~\ref{tecnico}, we deduce the truth of Eq. \eqref{final}.
\end{proof}

\section{Adelization of modular forms and Hecke operators}\label{s:background}

In this short section we briefly introduce some useful notation and make explicit our normalizations for modular forms and Hecke operators in the adelic setting. 

For simplicity, we set, in this section, $\G=\GL_2$ considered as an algebraic group over~$\Q$. We denote by~$\A$ the ring of ad\`eles of~$\Q$. Let 
$$\G(\R)^+=\{\gamma\in\G(\R)\colon \det\gamma>0\}.$$
For each prime number~$p$, we denote by $\iota_p\colon\G(\Q)\rightarrow\G(\A)$ the map induced by the ring homomorphism~$\Q\hookrightarrow\Q_p\rightarrow\A$. We define similarly $\iota_\infty\colon\G(\R)\rightarrow\G(\A)$ using the inclusion~$\R\hookrightarrow\A$. We then embed $\G(\Q)$ in~$\G(\A)$ diagonally (that is, using $\prod_p\iota_p\times\iota_\infty$)  and we embed $\G(\R)^+$  at infinity (that is, using~$\iota_\infty$).

Let $N\geq1$ be a positive integer. For every prime number~$p$ set
\[
K_p(N)=\left\{\begin{pmatrix}
               a & b \\
	      c & d \\
              \end{pmatrix}\in\G(\Z_p)\colon c\equiv0\pmod{N\Z_p}
\right\}
\]
and define $K_0(N)=\prod_{p}K_p(N)$ as a subgroup of~$\G(\A_\f)$ where $\A_\f$ denotes the finite ad\`eles of~$\Q$. The strong approximation theorem (\cite[Thm.~3.3.1]{Bum97} for~$\G$ then implies that
\begin{equation}\label{eq:SAT}
\G(\A)=\G(\Q)\G(\R)^+K_0(N).
\end{equation}

We denote by~$\omega$ the adelization (\emph{loc. cit.} Prop.~3.1.2) of a given Dirichlet character~$\chi$ of modulus~$N$, and define the group homomorphism
\[
 \begin{array}{cccl}
\lambda\colon&  K_0(N) & \longrightarrow & \C^\times \\
& \left(\begin{pmatrix}
a_p & b_p \\
c_p & d_p \\
\end{pmatrix}\right)_p &
\longmapsto & 
\displaystyle{\prod_{p\mid N}\omega_p(d_p)}
\end{array}.
\]

Let $p$ be a prime divisor of~$N$. For every integer $n\in\{0,\ldots,p-1\}$, define 
\[
\xi_n=\begin{pmatrix}
p & n \\
0 & 1 \\
\end{pmatrix}.
\]
Let  $k_0\in K_0(N)$. Denote by $\begin{pmatrix}
a & b \\
c & d \\
\end{pmatrix}\in K_p(N)$ the $p$-th component of~$k_0$. Let $n\in\{0,\ldots,p-1\}$ be an integer. Since $p\mid N$, we have that $cn+d\in\Z_p^\times$ and we define $m$ to be the unique integer in~$\{0,\ldots,p-1\}$ such that
\[
(cn+d)m\equiv an+b\pmod{p\Z_p}.
\]
Let $k_0'=\iota_p(\xi_m)^{-1}k_0\iota_p(\xi_n)$.  It  follows from the following matrix identity in~$\G(\Q_p)$
$$
\xi_m^{-1}\begin{pmatrix}
a & b \\
c & d \\
\end{pmatrix}\xi_n=\begin{pmatrix}
a-mc & \frac{an+b-m(cn+d)}{p} \\
cp & cn+d
\end{pmatrix}
$$
that 
\begin{equation}\label{eq:matrix_id}
k_0'\in K_0(N)\quad\text{and}\quad \lambda(k_0')=\lambda(k_0).
\end{equation}

Let $g \in \G(\A)$, that we decompose as 
$$
g=\gamma g_\infty k_0, \quad \gamma \in \G(\Q), \quad g_\infty \in \G(\R)^+, \quad k_0\in K_0(N)
$$
using Eq.~(\ref{eq:SAT}). We then check place by place that the following equality holds (see~\emph{loc. cit.}, p.~345)
\begin{equation}\label{eq:commute}
g\iota_p(\xi_n)=\left(\gamma\xi_m\right)\left(\xi_{m,\infty}^{-1}g_\infty\right)\left(\xi_{m,\f}^{-1}\iota_p(\xi_m)k_0'\right)\in\G(\Q)\G(\R)^+K_0(N)
\end{equation}
where $n\in\{0,\ldots,p-1\}$ and $m\in\{0,\ldots,p-1\}$, $k_0'\in K_0(N)$ are defined above. Here, $\xi_{m,\f}$ and $\xi_{m,\infty}$ denote the finite and the infinite components of~$\xi_m\in\G(\A)$ respectively.

Let $k\geq2$ be an integer. Denote by $S_k\left(N,\chi\right)$ the space of cuspidal modular forms of weight $k$, level~$N$ and Nebentypus character~$\chi$. To a mo\-dular form $F \in S_k\left(N,\chi\right)$, we attach 
$$
\phi_F\colon \G(\A) \rightarrow \C, \quad \phi_F(g)=F(g_\infty\cdot i)j(g_\infty, i)^{-k}\lambda(k_0),\quad g=\gamma g_\infty k_0.
$$ 
Here, for $g_\infty = \begin{pmatrix}
a & b \\
c &d 
\end{pmatrix}\in \G(\R)^+$, we have $j(g_\infty,z)=(cz+d)\det g_\infty^{-1/2}$. Since 
$$
\G(\Q)\cap\G(\R)^+ K_0(N)=\Gamma_0(N)=\left\{\begin{pmatrix}a&b\\c&d\\ \end{pmatrix}\in\SL_2(\Z)\colon c\equiv0\pmod{N}\right\}
$$ 
and for every $\gamma=\begin{pmatrix}a&b\\c&d\\ \end{pmatrix}\in\Gamma_0(N)$, we have $\lambda(\gamma)=\chi(d)^{-1}$ (as $\omega$ is trivial on~$\Q^\times$), the function~$\phi_F$ is a well-defined automorphic form (\emph{loc. cit.}, \S3.6). Define $\pi_F$ to be the linear span of right translates of~$\phi_F$ under~$\G(\A)$ and assume that $F$ is an eigenfunction for the Hecke operators away from~$N$. Then $\pi_F$ decomposes as a restricted tensor product~$\bigotimes'\pi_{F,v}$ where $v$ runs over the places of~$\Q$ and $\pi_{F,v}$ is an admissible irreducible representation of~$\G({\Q_v})$ (\emph{loc. cit.}, \S3.3). We now define the $p$-th Hecke operator in this adelic setting as follows (note the factor $1/\sqrt{p}$)
\begin{equation}\label{eq:Hecke_op}
\widetilde{U_p}=\frac{1}{\sqrt{p}}\sum_{n=0}^{p-1}\pi_{F,p}(\xi_n).
\end{equation}
The following result will be used in the proof of Theorems~\ref{main} and~\ref{main2}.
\begin{lemma}\label{lem:eigenvalue}
Let $U_p$ denote the $p$-th Hecke operator acting on~$S_k(N,\chi)$. Then, we have
\[
p^{\frac{k-1}{2}}\widetilde{U_p}\phi_F=\phi_{U_pF}.
\]
\end{lemma}
\begin{proof}
Let $g=\gamma g_\infty k_0\in\G(\A)$. Then, in the notation of Eq.~\eqref{eq:commute}, we have (using the fact that the map $n\mapsto m$ is a bijection of~$\{0,\ldots,p-1\}$)
\begin{align*}
\widetilde{U_p}\phi_F(g) &= \frac{1}{\sqrt{p}}\sum_{n=0}^{p-1} \phi_F\left(g\iota_p(\xi_n)\right) \\
& = \frac{1}{\sqrt{p}}\sum_{m=0}^{p-1} \phi_F\left(\left(\gamma\xi_m\right)\left(\xi_{m,\infty}^{-1}g_\infty\right)\left(\xi_{m,\f}^{-1}\iota_p(\xi_m)k_0'\right)\right) \\
& = \frac{1}{\sqrt{p}}\sum_{m=0}^{p-1} F\left(\left(\xi_{m,\infty}^{-1}g_\infty\right)\cdot i\right) j\left(\xi_{m,\infty}^{-1}g_\infty, i\right)^{-k} \lambda\left(\xi_{m,\f}^{-1}\iota_p(\xi_m)k_0'\right). \\
\end{align*}
Besides, from the definition of~~$\xi_m$ and Eq.~(\ref{eq:matrix_id}), we have 
$$
\lambda\left(\xi_{m,\f}^{-1}\iota_p(\xi_m)k_0'\right)=\lambda(k_0')=\lambda(k_0),
$$
and from the automorphy relation for~$F$, we have
$$
F\left(\left(\xi_{m,\infty}^{-1}g_\infty\right)\cdot i\right) j\left(\xi_{m,\infty}^{-1}g_\infty, i\right)^{-k}=p^{-k/2}F\left(\frac{g_\infty \cdot i-m}{p}\right)j(g_\infty,i)^{-k}.
$$ 
We conclude that 
\[
\widetilde{U_p}\phi_F(g) = \frac{1}{p^{(k+1)/2}}\sum_{m=0}^{p-1} F\left(\frac{g_\infty \cdot i-m}{p}\right)j(g_\infty,i)^{-k} \lambda(k_0).
\]
Hence, the desired identity follows from the formula
\[
U_pF(z)=\frac{1}{p}\sum_{m=0}^{p-1}F\left(\frac{z+m}{p}\right)
=\frac{1}{p}\sum_{m=0}^{p-1}F\left(\frac{z-m}{p}\right),\quad z\in\h.
\]
\end{proof}

\section{Proofs of the main results}\label{s:proofs}

\subsection{Proof of Theorem \ref{main}}

Let
\(\nu_1,\nu_2\colon\GQ\rightarrow\Ffbar^\times\) be characters such that $\rho=\nu_1\oplus\nu_2$ defines an odd (semi-simple) Galois representation of Serre type~\((N,k,\eps)\). Assume throughout that~$l>k+1$. Each of the characters $\nu_i$ ($i=1,2$) can be decomposed as $\nu_i=\eps_i\chi_l^{a_i}$ where $\eps_i$ is unramified at~$l$, $a_i$ is a non-negative integer and~$\chi_l$ denotes the mod~$l$ cyclotomic character. Without loss of generality, we may further assume that \(0\le a_1\le a_2\le l-2\). According to Serre's definition of the weight~$k$ (see~\cite[(2.3.2)]{Ser87}), we then have~:
\[
k=\left\{
\begin{array}{ll}
1+l a_1+a_2 &\text{if $(a_1,a_2)\not=(0,0)$} \\
l & \text{if $(a_1,a_2)=(0,0)$}
\end{array}
\right..
\]
Since we have assumed $l>k+1$, it follows that $(a_1,a_2)=(0,k-1)$. This proves the first part of Theorem~\ref{main}. 

Let us then  prove the equivalence. Denote by $\cond{1}$ and $\cond{2}$ the conductors of $\eps_1$ and $\eps_2$ respectively.  We have the Serre parameters $\eps=\eps_1\eps_2$ and $N=\cond{1}\cond{2}$. If $(N,k)=(1,2)$, then both $\eps_1$ and $\eps_2$ are trivial and therefore, in the notation of the theorem, we have
\[
B_{2,\eta}=B_2\pmod{l}\quad\text{and}\quad B_2=\frac{1}{6}\not\equiv 0\pmod{l}.
\]
On the other hand, there is no non-zero cuspidal eigenform of weight~$2$ and level~$1$ over~$\Ffbar$ for~$l\ge5$. Hence, the desired equivalence is established in this case.

From now on, let us then assume that either $N>1$ or $k>2$. Fix a place $\place$ of $\Qbar$ above~$l$ and denote by $\chi_1$ and $\chi_2$ the multiplicative lifts with respect to~$\place$ (in the sense of  paragraph~\ref{ss:Bernoulli_mod_l}) of $\eps_1$ and $\eps_2$ respectively. We view $\chi=\chi_1\chi_2$ as a Dirichlet character modulo~$N$. The Eisenstein series $E_k^{\chi_1,\chi_2}$ introduced in paragraph~\ref{ss:setting} (which is well-defined as we have $(N,k)\not=(1,2)$) has weight~$k$, level~$N$ and Nebentypus character~$\chi$. Moreover, it is a normalized eigenform for the full Hecke algebra at level~$N$. In particular, if we write
\[
E_k^{\chi_1,\chi_2}(z)=\sum_{n\geq0}a_n\left(E_k^{\chi_1,\chi_2}\right)e^{2i\pi zn}, \quad (z\in\h)
\]
then its eigenvalue for the action of the Hecke operator at an arbitrary prime~$p$  is given by
\[
a_p\left(E_k^{\chi_1,\chi_2}\right)=\chi_1(p)+\chi_2(p)p^{k-1}.
\]
By assumption, there exists an eigenform $f$ of type~$(N,k,\eps)$ over~$\Ffbar$ such that, in the notation of the Introduction, we have $\rho_f\simeq\rho$. Let us write $f=\sum_{n\geq1}a_nq^n$ as in~\cite[D\'ef. p.~193]{Ser87}. In other words, there exists $F=\sum_{n\geq1}A_nq^n$ a weight-$k$ cuspidal form of level~$N$ and Nebentypus character~$\chi$ such that~$A_n\in\Zbar_\place$ and
\begin{equation}\label{eq:reduction}
\nu_\place(A_n)=a_n,\quad\text{for any integer~$n\geq1$},
\end{equation}
in the notation of paragraph~\ref{ss:Bernoulli_mod_l}. By Deligne-Serre lifting lemma (\cite[Lem.~6.11]{DeSe74}), one may further assume that $F$ is a normalized eigenform for all the Hecke operators at level~$N$. Denote by $E$ the number field generated by the Hecke eigenvalues of $F$ and by $\lambda$ the prime ideal above~$l$ in~$E$ induced by~$\place$. Let $E_{\lambda}$ be the completion of~$E$ at~$\lambda$. Thanks to the isomorphism $\rho\simeq\rho_f$ and~\eqref{eq:reduction}, the semisimplification of the reduction modulo~$\lambda$ of the $\lambda$-adic representation of~$F$
\[
\rho_{F,\lambda}\colon\GQ\longrightarrow\GL_2(E_{\lambda}),
\]
is isomorphic to $\rho$.   Since $F$ has level~$N$ and $\rho$ conductor~$N$ away from~$l$, the form~$F$ is actually a newform. For every prime~$p\nmid Nl$, we have
\[
\nu_\place(A_p)=\eps_1(p)+\eps_2(p)p^{k-1},
\]
where, $\eps_i(p)=\eps_i(\Frob_p)$ if $\eps_i$ is unramified at~$p$ and $\eps_i(p)=0$ otherwise, for $i=1,2$. The next step is to extend these congruences to arbitrary primes~$p\not=l$, as stated in the following key result. (Note that only the case $N>1$ requires a proof.) 
\begin{prop}\label{prop:congruences}
In this notation, we have 
\[
\nu_\place(A_p)=\eps_1(p)+\eps_2(p)p^{k-1},\quad\text{for every prime~$p\not=l$.}
\]
\end{prop}
\begin{proof}We have seen that the equality holds for primes not dividing~$Nl$. Let~$p$ be a prime dividing~$N$ (note that, by definition, $N$ is coprime to~$l$ and hence $p\not=l$). We  denote by~$\cond{}$ the conductor of~$\chi$. We shall split the proof into  three cases :
\begin{enumerate}
\item\label{item:case_1} $\ord{p}(N)=1$ and $\ord{p}(\cond{})=0$;
\item $\ord{p}(N)\ge2$ and $\ord{p}(\cond{})<\ord{p}(N)$;
\item $\ord{p}(N)=\ord{p}(\cond{})$.
\end{enumerate}
To deal with the first two cases, we first observe that if  $\ord{p}(\cond{})<\ord{p}(N)$, then both characters $\chi_1$ and $\chi_2$ are ramified at~$p$. Indeed, since $\ord{p}(N)>0$ and $N=\cond{1}\cond{2}$, at least one of the two characters $\chi_1$ and $\chi_2$ is ramified at~$p$. On the other hand, if the other one is unramified at~$p$ then, we have
\[
\ord{p}(\cond{})=\ord{p}(\cond{1})+\ord{p}(\cond{2})=\ord{p}(N),
\]
obtaining a contradiction.

In the first case, using this observation, we obtain $$1=\ord{p}(N)=\ord{p}(\cond{1})+\ord{p}(\cond{2})\ge2$$ and a contradiction. Case~(\ref{item:case_1}) therefore does not occur.

In the second case, we have that $A_p=0$ (\cite[Thm.~4.6.17]{Miy06}) and by the above observation, both $\chi_1,\chi_2$ (and hence $\eps_1$ and $\eps_2$) are ramified at~$p$. We therefore have the desired equality as both sides are zero.

It therefore remains to deal with the last case. Let $\phi_F$ be the adelization of~$F$ as defined in Section~\ref{s:background}. Denote by~$\pi_F$ the corresponding automorphic representation. Since $F$ is $p$-new, then $\phi_F$ is a so-called new-vector for~$\pi_{F,p}$ (in the sense of~\cite[Thm.~2.2]{LoWe12}). The endomorphism~$\widetilde{U_p}$ defined in Eq.~(\ref{eq:Hecke_op}) acts on  the (one-dimensional) vector space of new-vectors of~$\pi_{F,p}$ by multiplication by an  eigenvalue that we denote by $\lambda(\pi_{F,p})$. It then follows from Lemma~\ref{lem:eigenvalue} that we have  
$$
\lambda(\pi_{F,p})=A_p/p^{(k-1)/2}.
$$

Using the assumption $\ord{p}(N)=\ord{p}(\cond{})$, we have that $\lambda(\pi_{F,p})$ has absolute value~$1$ and therefore is~$\not=0$ (\cite[Thm.~4.6.17]{Miy06}). On the other hand, we see from the classification of irreducible admissible infinite-dimensional smooth representations of $\GL_2(\Q_p)$ (as recalled in Table~1 of~\cite{LoWe12} for instance) that in this case $\pi_{F,p}$ necessarily is a principal series $\pi(\mu_1,\mu_2)$ associated with some characters $\mu_1,\mu_2$ of~$\Q_p^\times$. Equating the Hecke eigenvalues we find that
\begin{equation}\label{eq:eigenvalue}
p^{(k-1)/2}(\mu_1^*(p)+\mu_2^*(p))=A_p,
\end{equation}
where 
\[
\mu_i^*(p)=\left\{
\begin{array}{ll}
\mu_i(p) & \text{if $\mu_i$ is unramified at~$p$} \\
0 & \text{otherwise}
\end{array}
\right.,\quad\text{for }i=1,2.
\]
Let $\sigma^{\lambda}(\pi_{F,p})$ be the representation of the local Weil group~$W(\Qbar_p/\Q_p)$ attached to~$\pi_{F,p}$ by the local Langlands correspondence. By a theorem of Carayol (\cite[Thm.~(A)]{Car86}), it agrees with (the restriction to the Weil group of) the local representation $\rho_{F,\lambda}|_{\Gal(\Qbar_p/\Q_p)}$. 

Let us denote by $\overline{\mu_1}$ and $\overline{\mu_2}$ the reductions modulo~$\place$ of $\mu_1$ and $\mu_2$ respectively. According to~\S0.5 in~\emph{loc. cit.}, we therefore have the following equality of characters of $\Q_p^{\times}$ with values in~$\Ffbar^{\times}$~:
\[
\left\{\overline{\mu_1}\cyclomod^{(k-1)/2},\overline{\mu_2}\cyclomod^{(k-1)/2}\right\}=\left\{\eps_1,\eps_2\cyclomod^{k-1}\right\}.
\]
The result now follows from~(\ref{eq:eigenvalue}).

\end{proof}

Let us now consider the Eisenstein series $E_k^{\chi_1,\chi_2}$. Since both $F$ and $E_k^{\chi_1,\chi_2}$ are eigenfunctions for the full Hecke algebra at level~$N$, it follows from the previous proposition and the multiplicativity of the Fourier coefficients that
\[
\nu_\place(A_n)= \nu_{\place}\left(a_n\left(E_k^{\chi_1,\chi_2}\right)\right), \quad\text{for all prime-to-\(l\) integers \(n\)}.
\]
Note that by Lemma~\ref{lem:Carlitz} and Eq.~(\ref{eq:ES}), the $q$-expansion of the Eisenstein series $E_k^{\chi_1,\chi_2}$ lies in~$\overline{\Z}_\place[[q]]$. Let us denote by $\overline{E}$ its reduction modulo $w$. 
 Then, both \(f\) and $\overline{E}$ have the same image under the \(\Theta\)-operator whose action on the $q$-expansions is $q\frac{d}{dq}$ (see~\cite[Ch.~II]{Kat77}).
 
 We remark that, since we are assuming that $k \geq 2$ and  $l \nmid N$, the space of modular forms for $\Gamma_1(N)$ over $\Ffbar$ in the sense of Katz and in the sense of Serre are naturally isomorphic (\cite{DiIm93}, Theorem 12.3.7). Then, since \(l>k+1\), we can use \cite[Cor.~3]{Kat77} to assert  that the $\Theta$-operator is injective. Hence, $\overline{E}$ is a cuspidal form over~$\Ffbar$. This implies that \(\place\) divides the constant term of~\(E_k^{\chi_1,\chi_2}\) at each of the cusps. 

In particular, it divides the constant term of the Fourier expansion at~\(\infty\) of~\(E_k^{\chi_1,\chi_2}|_{k}\gamma\) where~\(\gamma=
\begin{pmatrix}
          1 & 0 \\
	  \cond{2} & 1 \\
\end{pmatrix}\in\SL_2(\Z)
\). According to Proposition~\ref{prop:cst_terms}  (applied to~\(M=1\)  in its notation), \(\place\) divides
\[
\left(\frac{\cond{2}}{\cond{0}}\right)^k
\frac{W\left((\chi_1\overline{\chi_2})_0\right)}{W\left(\overline{\chi_2}\right)}
\frac{B_{k,(\overline{\chi_1}\chi_2)_0}}{2k}
\prod_{p\mid N}\left(1-\left(\chi_1\overline{\chi_2}\right)_0(p)p^{-k}\right).
\]
However~\(\cond{0}\), \(\cond{2}\), \(W\left((\chi_1\overline{\chi_2})_0\right)\), $2k$ and~\(W\left(\overline{\chi_2}\right)\) are all coprime to~\(l\). Moreover, $(\overline{\chi_1}\chi_2)_0$ is nothing but the multiplicative lift of~$\eta=\eps_1^{-1}\eps_2$ with respect to~$\place$. Hence, either $B_{k,\eta}=0$, or there exists a prime~$p\mid N$ such that $\eta(p)p^k=1$. This proves the direct  implication in Theorem \ref{main}.

Conversely, assume that either condition of the theorem is satisfied. Then, by definition of the characters~$\chi_1$ and~$\chi_2$ and of the Bernoulli number~$B_{k,\eta}$, the place~$\place$ divides (the numerator of) 
\[
B_{k,(\overline{\chi_1}\chi_2)_0}
\cdot\prod_{p\mid N}\left((\overline{\chi_1}\chi_2)_0(p)p^{k}-1\right).
\]
Then, according to Proposition~\ref{prop:cst_terms} (with~\(M=1\)), the constant term of the Eisenstein series ~\(E_k^{\chi_1,\chi_2}\) vanishes at each of the cusp of the modular curve~$X_1(N)$. Let $f$ be  its reduction modulo~\(\place\), which is an eigenform with coefficients in~\(\Ffbar\). As we argued before, $f$ can be seen both as a Katz or Serre modular form. Then, the $q$-expansion principle allows us to ensure that $f$ is a cuspidal eigenform (\emph{cf.} \cite{DiIm93}, Remark 12.3.5).

On the other hand, for every prime~$q\nmid Nl$, we have
\[
\textrm{trace}\left( \rho_f(\Frob_q)\right)=\nu_\place\left(a_q(E_k^{\chi_1,\chi_2})\right)=\eps_1(q)+\eps_2(q)q^{k-1}=\textrm{trace}\left(\rho(\Frob_q)\right).
\]

Since $\det\rho_f=\eps\cyclomod^{k-1}=\det\rho$, the Chebotarev Density and Brauer-Nesbitt theorems, as explained in \cite[Lem.~3.2]{DeSe74}, imply that  $\rho_f\simeq\rho$. Then, $f$ is the desired eigenform. This finishes the proof of Theorem~\ref{main}.

\subsection{Proof of Theorem \ref{main2}}\label{conducteur}

In the case $(N,k)=(1,2)$ (where we necessarily have $\rho\simeq\trivial\oplus\cyclomod$ and hence~\(\rho\) is not strongly modular), the result is due to Mazur (\cite[Prop.~5.12]{Maz77}). 

We therefore assume throughout that $(N,k)\not=(1,2)$ and start by proving the direct implication.

Using the assumption that the representation~$\rho$ arises from a modular form~$f$ of type~$(NM,k,\varepsilon)$ over~$\Ffbar$, we show as before that there exists $F=\sum_{n\geq1}A_nq^n$, a weight-$k$ normalized cuspidal eigenform of level~$NM$ and Nebentypus character~$\chi$ with the following property. Let $\lambda$ be  the  prime ideal of the coefficient field of~$F$ induced by~$\place$. The semisimplification of the reduction modulo~$\lambda$ of the $\lambda$-adic representation attached to~$F$  is isomorphic  to~$\rho$. Let $F_0$ denote the newform associated with~$F$. The $\lambda$-adic representations attached to~$F_0$ and~$F$ are isomorphic. In particular, after reduction modulo~$\lambda$ and semisimplification, they both give rise to~$\rho$. Since $\rho$ has conductor~$N$, it follows from~\cite[Thm.~(A)]{Car86} and the considerations in  \cite[{\bf 1.-2.}]{Car89},  that the level of~$F_0$ is divisible by~$N$. Moreover, we have assumed that $\rho$ is not strongly modular, and thus the level of~$F_0$ is strictly greater than~$N$. Since it is a divisor of~$NM$, it has to be equal to~$NM$ and $F=F_0$ necessarily is a newform. Therefore, considering its associated automorphic representation, we prove the following result using the same arguments as in Proposition~\ref{prop:congruences}.
\begin{prop}\label{prop:away_from_M}
In this notation, we have 
\[
\nu_\place(A_p)=\eps_1(p)+\eps_2(p)p^{k-1},\quad\text{for every prime~$p\not=l,M$.}
\]
\end{prop}
We now turn our attention to the local situation at~$M$ and prove the following statement.
\begin{prop}\label{prop:at_M}
We have 
\begin{enumerate}
\item\label{item:cond1} either $\eta(M)M^k=1$;
\item\label{item:cond2} or, $\eta(M)M^{k-2}=1$ and $\nu_\place(A_M)=\eps_1(M)$.
\end{enumerate}
\end{prop}
\begin{proof}
According to~\cite[Thm.~4.6.17(2)]{Miy06}, we have $A_M\not=0$. In parti\-cular, the form~$F$ is $M$-primitive in the sense of~\cite[Def. p.~236]{AtLi78} (see the remark right after the definition). Therefore, according to Proposition~2.8 of~\cite{LoWe12}, the local component at~$M$ of the automorphic represention of~$F$ corresponds to a Steinberg representation. Moreover, we have the following equality between sets of characters of a decomposition group at~$M$ in~$\Gal(\Qbar/\Q)$ with values in~$\Ffbar^\times$~:
\[
\left\{\eps_1,\eps_2\cyclomod^{k-1}\right\}=\left\{\mu\cyclomod^{k/2},\mu\cyclomod^{k/2-1}\right\},
\]
where $\mu$ is the unramified character that sends a Frobenius element at~$M$ to $\mu(M)=\nu_\place\left(A_M/M^{k/2-1}\right)$. We therefore have two cases to consider~:
\begin{itemize}
\item Assume that, locally at~$M$, we have $\eps_1=\mu\cyclomod^{k/2}$. Then, in particular, we have \(\eps_1(M)^2=\mu(M)^2M^k\). On the other hand, according to~\cite[Thm.~4.6.17]{Miy06}, we have $\mu(M)^2=(\eps_1\eps_2)(M)$. Therefore, we get that $\eta(M)M^k=1$. (Note that the other equality, namely $\eps_2\cyclomod^{k-1}=\mu\cyclomod^{k/2-1}$, does not provide any additional information.)
\item Assume instead that, locally at~$M$, we have \(\eps_1=\mu\cyclomod^{k/2-1}\). Then, on the one hand, we have that $\eps_1(M)=\mu(M)M^{k/2-1}$ and hence $\nu_\place(A_M)=\eps_1(M)$. On the other hand, we have (using \emph{loc. cit.}) $M^{2k-2}\eps_2(M)^2=\mu(M)^2M^k$. Therefore we get that $\eta(M)M^{k-2}=1$. Hence the result follows. (Once again, the other equality, namely \(\eps_2\cyclomod^{k-1}=\mu\cyclomod^{k/2}\), does not give any other information.)
\end{itemize}
\end{proof}
In order to finish the proof of Theorem~\ref{main2}, it therefore remains to show that, under the assumption that~$\rho$ is not strongly modular, condition \eqref{item:cond2} in Proposition \ref{prop:at_M} implies condition \eqref{item:cond1}. For that purpose, let us assume that condition~(\ref{item:cond2}) is satisfied and consider the following Eisenstein series~:
\[
F_1=E_k^{\chi_1,\chi_2}-\chi_2(M)M^{k-1}\alpha_ME_k^{\chi_1,\chi_2}.
\]
It is a well-known fact that $F_1$ is an eigenform for the full Hecke algebra at level~$NM$ with eigenvalues
\[
a_p(F_1)=\chi_1(p)+\chi_2(p)p^{k-1},\quad\text{for primes }p\not=M
\]
and $a_M(F_1)=\chi_1(M)$. In particular, as a consequence of Proposition~\ref{prop:away_from_M} and our assumption, we have
\begin{equation}\label{eq:reduction_of_F_1}
\nu_\place\left(a_n(F_1)\right)=\nu_\place(A_n),\quad\text{for every integer~$n$ coprime to~$l$},
\end{equation}
where $\left\{a_n(F_1)\right\}_{n\geq1}$ denote the coefficients of the Fourier expansion of~$F_1$ at~$\infty$.  By definition of~$F_1$, Lemma~\ref{lem:Carlitz} and Eq.~(\ref{eq:ES}), this $q$-expansion lies in~$\overline{\Z}_\place[[q]]$. Let us thus denote by $\overline{F_1}$ the reduction of~$F_1$ modulo~\(\place\). According to~(\ref{eq:reduction_of_F_1}), $\overline{F_1}$ and the reduction of~$F$ modulo~$\place$ have the same image under the $\Theta$-operator. Since $l>k+1$, the injectivity of $\Theta$ (\cite[Cor.~3]{Kat77}) implies that $\overline{F_1}$ is cuspidal. Therefore, we have that $\place$ divides the numerator of the constant of the term of the Fourier expansion of~$F_1$ at each cusp of the modular curve at level~$NM$. According to Corollary~\ref{cor:cst_terms}, such a constant term at the cusp~$1/(M\cond{2})$ is given (up to  roots of unity) by 
\[
\Upsilon_k^{\chi_1,\chi_2}(\gamma,1)\left(1-\overline{\chi_1}(M)\chi_2(M)M^{k-1}\right),
\]
where $\gamma\in\SL_2(\Z)$ is such that $\gamma\cdot\infty=1/(M\cond{2})$. On the other hand, for such a $\gamma$, thanks to Theorem~\ref{main} and Proposition~\ref{prop:cst_terms}, the assumption that $\rho$ is not strongly modular guarantees that $\Upsilon_k^{\chi_1,\chi_2}(\gamma,1)$ is (non-zero and) not divisible by~$\place$. Therefore, it follows  that $\eta(M)M^{k-1}=1$ and hence $M\equiv1\pmod{l}$ (as we have assumed $\eta(M)M^{k-2}=1$). This implies the desired equality $\eta(M)M^k=1$ and concludes the proof of the direct implication.

In the other direction, assuming that $\eta(M)M^k=1$, we now consider the Eisenstein series defined by
\[
F_2=E_k^{\chi_1,\chi_2}-\chi_1(M)\alpha_ME_k^{\chi_1,\chi_2}.
\]
For any $\gamma\in\SL_2(\Z)$, let us denote by $a_0\left(F_2|_k \gamma\right)$ the constant term of the Fourier expansion at~$\infty$ of~$F_2|_k \gamma$. According to Corollary~\ref{cor:cst_terms}, using its notation, we have that
\[
a_0\left(F_2|_k \gamma\right)=\Upsilon_k^{\chi_1,\chi_2}(\gamma,1)
\left(1-\left(\frac{r}{M}\right)^k(\chi_1\overline{\chi_2})(M/r)\right),
\]
where $r=1$ or~$M$. In both cases, using the assumption $\eta(M)M^k=1$, we have that $\nu_\place\left(a_0\left(F_2|_k \gamma\right)\right)=0$. We denote by $f$ the reduction of~$F_2$ modulo~$\place$. It is a well-defined cuspidal form of type~$(NM,k,\eps)$ over~$\Ffbar$ which is an eigenform for the full Hecke algebra at level~$NM$ with eigenvalue for the Hecke operator at~$p$ given by
\[
\eps_1(p)+\eps_2(p)p^{k-1},\quad \text{for all primes }p\not=M.
\]
Then, the Chebotarev Density and Brauer-Nesbitt theorems, as explained in \cite[Lem.~3.2]{DeSe74}, imply  that $\rho$ arises from a form of type~$(NM,k,\eps)$ as desired.

%

\end{document}